\documentclass[11pt]{amsart}

\usepackage{amsthm}
\usepackage{amsmath}
\usepackage{amssymb}

\newtheorem{thm}{Theorem}[section]
\newtheorem{theorem}[thm]{Theorem}

\newtheorem{corollary}[thm]{Corollary}
\newtheorem{problem}[thm]{Problem}

\newtheorem{lemma}[thm]{Lemma}
\newtheorem{proposition}[thm]{Proposition}

\newtheorem{definition}[thm]{Definition}
\theoremstyle{remark}

\newcommand{\norm}[1]{\|#1\|}

\newcommand{\cH}{\mathcal H}

\newcommand{\cF}{\mathcal F}
\newcommand{\cG}{\mathcal G}

\title{ the Paulsen Problem in operator theory}
\author[J. Cahill and P.G. Casazza 
 ]{Jameson Cahill and Peter G. Casazza}
\address{Department of Mathematics, University
of Missouri, Columbia, MO 65211-4100}

\thanks{The authors were supported by NSF DMS 1008183,
DTRA/NSF 1042701, AFOSR F1ATA00183G003}

\email{jameson.cahill@gmail.com; casazzap@missouri.edu}

\begin{document}

\begin{abstract}
The {\it Paulsen Problem} in Hilbert space frame theory has proved to
be one of the most intractable problems in the field.  We will help
explain why by showing that this problem is equivalent to a fundamental,
deep problem in operator theory.  This answers a question posed
by Bodmann and Casazza. 
We will also give generalizations
of these problems and we will spell out exactly the {\it complementary
versions} of the problem.
\end{abstract}

\maketitle


\section{Introduction}

The {\it Paulsen Problem} has proved to be one of the most intractable problems
in frame theory (See Section \ref{sec2} for definitions):

\begin{problem}[Paulsen Problem]
Find the function $h(\epsilon,M,N)$ so that for any $\epsilon$-nearly equal norm,
$\epsilon$-nearly Parseval frame $\{f_i\}_{i=1}^N$ for a $M$-dimensional Hilbert
space $\cH_M$, there is an equal norm Parseval frame $\{g_i\}_{i=1}^M$ 
for $\cH_M$ satisfying:
\[ \sum_{i=1}^N \|f_i-g_i\|^2 \le h(\epsilon,M,N).\]
\end{problem}

A fundamental question here is whether the function $h(\epsilon,M,N)$
actually depends upon $N$.  We have no examples showing this at this time, although
it is known that this function must depend upon $M$.  For all examples we
know at this time, we have
\[ f(\epsilon,M,N) \le 16 \epsilon M.\]

For a dozen years no progress at all was made on the Paulsen Problem.  Recently,
some progress has been made on the problem.  First, Bodmann and Casazza 
\cite{BC} used differential equations to give an estimate for the function $h(\epsilon,M,N)$.
This paper leaves open the case where $M,N$ are not
relatively prime.  Using gradient descent of the {\it frame potential}, Casazza, Fickus
and Mixon \cite{CF} gave a completely different solution for the Paulsen problem
which works in the case where $M,N$ are relatively prime.    
The estimates in these two papers seem to be quite far from optimal
since it is on the order of $M^2N^9\epsilon$ and best evidence indicates the
answer should be of the form $cM\epsilon$ or at worst $cN\epsilon$.

We will show why the Paulsen Problem has proved to so intractable by showing that
it is equivalent to a fundamental, deep problem in operator theory.  The fact
that there must be a connection
between these two problems was first observed in \cite{BC}.  In effect, we are answering
a problem left open in that paper.

\begin{problem}[Projection Problem]\label{prob1}
Let $\cH_N$ be an $N$-dimensional Hilbert space with orthonormal basis $\{e_i\}_{i=1}^N$.
Find the function $g(\epsilon,M,N)$ satisfying the following.  If $P$ is a projection of rank $M$
on $\cH_N$ satisfying
\[ (1-\epsilon)\frac{M}{N} \le \|Pe_i\|^2 \le (1+\epsilon)\frac{M}{N},
\mbox{ for all } i=1,2,\ldots,N,\]
then there is a projection $Q$ with $\|Qe_i\|^2 = \frac{M}{N}$ for all $i=1,2,\ldots,N$ 
satisfying
\[ \sum_{i=1}^N \|Pe_i-Qe_i\|^2 \le g(\epsilon,M,N).\]
\end{problem}

In \cite{BC}, it is shown that the Hilbert-Schmidt distance between an $\epsilon$-nearly
constant diagonal projection and its closest constant diagonal projection is less
than or equal to $2f(\epsilon,M,N)$.  Here, we will show the full equivalence (up to
a factor of 4) of $f(\epsilon,M,N)$ and $g(\epsilon,M,N)$.
Analyzing the diagonal properties of projections has a long history. 
Kadison \cite{K1,K2} gave a complete characterization of the diagonals
of projections for bothe the finite and infinite dimensional case.   
Analogous results on projections in type $II_1$ factors
was given by Argerami and Massey \cite{AM}.  For the more general problem of
characterizing the diagonals of the unitary orbit of a self-adjoint operator,
there is much more literature.  This is equivalent in frame theory to characterizing
the sequences which occur as the norms of a frame with a specified frame operator.
We refer the reader to \cite{AMRS,AM2,A,AK,BJ,CL1,CL,CFLKT,H,J,KL,KW,N,S}
for a review of the work in this direction.

We will also consider the {\it Naimark complement} of nearly equal norm 
Parseval frames.  We will show that the Paulsen function for a Parseval frame
and its Naimark complement have a natural relationship.  As a consequence of this,
we will see that the Paulsen Problem only has to be solved for frames with a
small number of elements relative to the dimension of the space.  In particular,
we only have to deal with the case of $N\le 2M$.

This paper is organized as follows.  In Section 2 we give the requisite background
needed from frame theory.  In Section 3 we will prove a sequence of results which
give an exact relationship between nearly equal norm Parseval frames for 
$\cH_M$ and the distance between orthogonal projections $P,Q$ on $\cH_N$
of rank $M$.  As a tool here, we will relate our quantities
to the {\it principal angles} between
subspaces of a Hilbert space and the {\it chordal distance} between subspaces of
a Hilbert space.  In Section 4 we give an exact calculation relating the Paulsen 
Problem function and the function in the Projection Problem.  Section 5 contains
generalizations of both problems and in Section 6 we will relate the Paulsen
Problem functions for a frame and its Naimark complement.

\section{Frame Theory}\label{sec2}

In this section we will give a brief introduction to frame theory containing the results
used in the paper.  For the basics on frame theory see \cite{Chr}. 

\begin{definition}
A family of vectors $\{f_i\}_{i=1}^N$ in an $M$-dimensional Hilbert space $\cH_M$ is
a {\it frame} if there are constants $0<A\le B < \infty$ so that for all $f\in \cH_M$ we have
\[ A \|f\|^2 \le \sum_{i=1}^M|\langle f,f_i \rangle |^2 \le B \|f\|^2.\]
If $A=B$, this is a {\it tight frame} and if $A=B=1$, it is a {\it Parseval frame}. If there is
a constant $c$ so that $\|f_i\|=c$, for all $i=1,2,\ldots,N$ it is an {\it equal norm frame} and
if $c=1$, it is a {\it unit norm frame}.
\end{definition} 

If $\{f_i\}_{i=1}^N$ is a frame for $\cH_M$, the {\it analysis operator} of the frame is
the operator $T:\cH_M \rightarrow \ell_2(N)$ given by
\[ T(f) = \sum_{i=1}^N \langle f,f_i\rangle e_i,\]
where $\{e_i\}_{i=1}^N$ is the natural orthonormal basis of $\ell_2(N)$.  The {\it synthesis
operator} is $T^*$ and satisfies
\[ T^*\left (\sum_{i=1}^Na_ie_i \right )= \sum_{i=1}^Na_if_i.\]
The {\it frame operator} is the positive, self-adjoint invertible operator $S=T^*T$
on $\cH_M$  and satisfies
\[ S(f)=: T^*T(f) = \sum_{i=1}^N \langle f,f_i\rangle f_i.\]
A direct calculation shows that the frame $\{S^{-1/2}f_i\}_{i=1}^N$ is a Parseval frame
called the {\it canonical Parseval frame} for the frame.  Also, $\{f_i\}_{i=1}^N$ is a 
Parseval frame if and only if $S=I$.
We say that two frames $\{f_i\}_{i\in I},\{g_i\}_{i\in I}$ for $\cH$ are {\it isomorphic} if there is an
invertible operator $L$ on $\cH$ satisfying $Lf_i=g_i$, for all $i\in I$.  It is known
\cite{CKo} that two frames are isomorphic if and only if their analysis operators have the same
image, and two Parseval frames are isomorphic if and only if the isomorphism is a
unitary operator.
If $\{f_i\}_{i=1}^N$ is a frame with frame operator $S$ having eigenvalues $\{\lambda_j\}_{j=1}^M$,
then
\[ \sum_{i=1}^N \|f_i\|^2 = \sum_{j=1}^M\lambda_j.\]
So if $\{f_i\}_{i=1}^N$ is an equal norm Parseval frame then
\[ \|f_1\|^2 = \frac{1}{N}\sum_{i=1}^N\|f_i\|^2 = \frac{M}{N}.\]

We will need a distance function for frames and projections.

\begin{definition}
If $\cF = \{f_i\}_{i=1}^N$ and ${\mathcal G}=\{g_i\}_{i=1}^N$ are frames for $\cH_M$, we define
the distance between them by
\[ d(\cF,\cG) = \sum_{i=1}^N\|f_i-g_i\|^2 .\]
If $P,Q$ are projections on $\ell_2(N)$, we define
\[ d(P,Q) = \sum_{i=1}^N \|Pe_i-Qe_i\|^2,\]
where $\{e_i\}_{i=1}^N$ is the natural orthonormal basis for $\ell_2(N)$.
\end{definition}

 For the Paulsen Problem, we define:
 
\begin{definition}
A frame $\{f_i\}_{i=1}^N$ with frame operator $S$ is $\epsilon$-nearly Parseval if
\[ (1-\epsilon)I \le S \le (1+\epsilon)I.\]
The frame is $\epsilon$-nearly equal norm if
\[ (1-\epsilon)\frac{M}{N} \le \|f_i\|^2 \le (1+\epsilon)\frac{M}{N}.
\]
\end{definition}

 A reduction of the Paulsen problem to the Parseval case is done in \cite{BC}.
 
 \begin{proposition}
 If $\cF=\{f_i\}_{i=1}^N$ is an $\epsilon$-nearly Parseval frame 
for $\cH_M$ then
the Parseval frame $\cG = \{S^{-1/2}f_i\}_{i=1}^N$ satisfies
\[ d(\cF,\cG) \le M(2 - \epsilon - 2\sqrt{1-\epsilon}) \le \frac{M\epsilon^2}{4}.
\]
It is also nearly equal norm
with the bounds:
\[ \frac{(1-\epsilon)^2}{1+\epsilon} \frac{M}{N} \le \|S^{-1/2}f_i\|^2
\le \frac{(1+\epsilon)^2}{1-\epsilon}\frac{M}{N}.\] 
\end{proposition}

It is known \cite{B,CK,J} that the canonical Parseval frame is the closest Parseval frame
(with the distance function above) to a given frame.  It is also known
that this constant is best possible in general.  So we are not giving up
anything by working with a simpler variation of the Paulsen Problem.

\begin{problem}[Parseval Paulsen Problem]
Find the function $f(\epsilon,M,N)$ so that whenever $\{f_i\}_{i=1}^N$ is an
$\epsilon$-nearly equal norm Parseval frame, then there is an equal norm
Parseval frame ${\mathcal G}$ so that
\[ d({\mathcal F},{\mathcal G}) \le f(\epsilon,M,N).\]
\end{problem}

Finally, we recall a fundamental result in frame theory -  the classification theorem for 
Parseval frames \cite{Chr,HL} - which will be used extensively here.

\begin{theorem}
A family $\{f_i\}_{i=1}^N$ is a Parseval frame for $\cH_M$ if and only if
the analysis operator $T$ for the frame is a co-isometry satisfying:
\[ Tf_i = Pe_i, \mbox{ for all } i=1,2,\ldots,N.,\]
where $\{e_i\}_{i=1}^N$ is the natural orthonormal basis of $\ell_2(N)$ and
$P$ is the orthogonal projection of $\ell_2(N)$ onto $T(\cH_M)$.
\end{theorem}

\section{Preliminary Results}

Let us first outline the proof of the equivalence of  the Paulsen Problem and the Projection
Problem.  This will explain the results we develop in this section. 

First we will
assume that the Parseval Paulsen Problem function $f(\epsilon,M,N)$ is given
and let $P$ be a rank $M$ projection on $\ell_2(N)$ with 
$\epsilon$-nearly constant
diagonal.  We need to find a constant diagonal projection whose distance to
$P$ is on the order of $f(\epsilon,M,N)$.  To do this, we consider ${\mathcal F}=\{Pe_i\}_{i=1}^N$
a nearly equal norm Parseval frame for $\cH_M$.  It follows that there is
a equal norm Parseval frame ${\mathcal G}=\{g_i\}_{i=1}^N$ for $\cH_M$ with
\[ d({\mathcal F},{\mathcal G}) \le f(\epsilon,M,N).\]
Letting $T_1$ be the analysis operator for $\mathcal G$, we have the existence of
a projection $Q$ on $\ell_2(N)$ so that
\[ T_1g_i = Qe_i,\mbox{ for all } i=1,2,\ldots,N.\]
So it is the problem of finding $d(P,Q)$ we will address in this section.

Conversely, if we assume the Projection Problem function $g(\epsilon,M,N)$ is given,
we choose a nearly equal norm Parseval frame ${\mathcal F}=\{f_i\}_{i=1}^N$
with analysis operator $T:\cH_M \rightarrow \ell_2(N)$ a co-isometry and satisfying
\[ Tf_i = Pe_i,\mbox{ for all } i=1,2,\ldots,N.\]
We need to find the closest equal norm Parseval frame to $\mathcal F$.
By our assumption, $P$ is a projection with nearly constant diagonal.  By the Projection Problem,
there is a projection $Q$ on $\ell_2(N)$ with $d(P,Q) \le g(\epsilon,M,N)$.  It follows
that $\{Qe_i\}_{i=1}^N$ is a equal norm Parseval frame.   We will be done if we can
find an equal norm Parseval frame ${\mathcal G}=\{g_i\}_{i=1}^N$ for $\cH_M$ with analysis
operator  $T_1$ satisfying:
\begin{equation}\label{eqn1}
T_1g_i = Qe_i,\mbox{ and } d({\mathcal F},{\mathcal G}) \approx g(\epsilon,M,N).
\end{equation}
So it is the problem of finding $\mathcal G$ we address in this section.  This problem is
made more difficult by the fact that there are many frames $G$ satisfying Equation
\ref{eqn1} and most of them are not close to 
$\mathcal F$.  In particular, if ${\mathcal G}=\{g_i\}_{i=1}^N$ satisfies Equation \ref{eqn1},
and $U$ is any unitary operator on $\cH_M$, then $U({\mathcal G}) = \{Ug_i\}_{i=1}^N$
also satisfies Equation \ref{eqn1}.  To address this problem, we will introduce the
{\it chordal distance} between subspaces of a Hilbert space and give a computation
of this distance in terms of our distance function.  Using this, we will be able to construct
the required frame $\mathcal G$.

We need a result
from \cite{BC} and for completeness include its proof.

\begin{theorem}\label{T3}
Let ${\mathcal F}=\{f_i\}_{i\in I},{\mathcal G}=\{g_i\}_{i\in I}$ be Parseval frames for $\cH$
with analysis operators $T_1,T_2$ respectively.  
If
\[ d({\mathcal F},{\mathcal G}) = \sum_{i\in I}\|f_i-g_i\|^2 < \epsilon,\]
then
\[ d(T_1({\mathcal F}),T_2({\mathcal G})) =
\sum_{i\in I}\|T_1f_i-T_2g_i\|^2 < 4\epsilon.\]
\end{theorem}

\begin{proof}
Note that for all $j\in I$,
\[ T_1f_j = \sum_{i\in I}\langle f_j,f_i\rangle e_i,\mbox{  and  } 
 T_2g_j = \sum_{i\in I}\langle g_j,g_i\rangle e_i.\]
Hence,
\begin{eqnarray*}
\|T_1f_j - T_2g_j\|^2 &=& \sum_{i\in I} |\langle f_j,f_i\rangle - \langle 
g_j,g_i\rangle |^2\\
&=& \sum_{i\in I}|\langle f_j,f_i-g_i\rangle + \langle f_j-g_j,g_i
\rangle |^2\\
&\le& 2 \sum_{i\in I}|\langle f_j,f_i-g_i\rangle |^2 + 2\sum_{i\in I}
|\langle f_j-g_j,g_i\rangle|^2.
\end{eqnarray*}
Summing over $j$ and using the fact that our frames $\mathcal F$ and $\mathcal G$
are Parseval gives
\begin{eqnarray*}
\sum_{j\in I}\|T_1f_j-T_2g_j\|^2 &\le& 2 \sum_{j\in I}\sum_{i\in I}
|\langle f_j,f_i-g_i\rangle|^2 + 2\sum_{j\in I}\sum_{i\in I}
|\langle f_j-g_j,g_i\rangle|^2\\
&=& 2\sum_{i\in I}\sum_{j\in I}|\langle f_j,f_i-g_i\rangle|^2
+2\sum_{j\in I}\|f_j-g_j\|^2\\
&=& 2\sum_{i\in I}\|f_i-g_i\|^2 + 2\sum_{j\in I}\|f_j-g_j\|^2\\
&=&4\sum_{j\in I}\|f_j-g_j\|^2.
\end{eqnarray*}
\end{proof}

 As we noted above,  $d(T_1({\mathcal F}),T_2({\mathcal G}))$
need not be bounded by $d({\mathcal F},{\mathcal G})$ in general. 
We now show that there is at least one choice of $\mathcal G$ which
gives the correct bound.  For this, we need to introduce 
{\it principle angles} and the {\it chordal distance}
between subspaces of a Hilbert space.  For notation, if $\cH$ is a Hilbert space,
denote the unit sphere by $Sp_{\cH}$.

\begin{definition}
Given $M$-dimensional subspaces $W_1,W_2$ of a Hilbert space, define the
$M$-tuple $(\sigma_1,\sigma_2,\ldots,\sigma_M)$ as follows:
\[ \sigma_1 = max\{\langle f,g\rangle: f\in Sp_{W_1},\ g\in Sp_{W_2}\}= \langle f_1,g_1\rangle.\]
For $2\le i \le M$,
\[ \sigma_i = max \{\langle f,g\rangle:\|f\|=\|g\|=1,\ \langle f_j,f\rangle = 0=\langle g_j,g\rangle,
\mbox{ for } 1\le j\le i-1\},\]
where
\[\sigma_i= \langle f_i,g_i\rangle.
\]
\end{definition}

The $M$-tuple $(\theta_1,\theta_2,\ldots, \theta_M)$ with $\theta_i = cos^{-1}(\sigma_i)$
is called the {\it principle angles} between $W_1,W_2$.  The {\it chordal distance}
between $W_1,W_2$ is given by
\[ d_c^2(W_1,W_2) = \sum_{i=1}^M sin^2\theta_i.\]
So by the definition, there exists orthonormal bases $\{a_j\}_{j=1}^M, \ \{b_j\}_{j=1}^M$
for $W_1,W_2$ respectively satisfying
\[ \|a_j-b_j\| = 2sin\ \left (\frac{\theta}{2}\right ),\mbox{ for all } j=1,2,\ldots,M.\]
It follows that for $0\le \theta \le \frac{\pi}{2}$,
\[ sin^2\theta \le 4 sin^2\ \left ( \frac{\theta}{2}\right )= \|a_j-b_j\|^2  \le 4 sin^2 \theta,
\mbox{ for all } j=1,2,\ldots,M.\] 
Hence,
\begin{equation}\label{E17}
 d_c^2(W_1,W_2) \le
\sum_{j=1}^M \|a_j-b_j\|^2 \le 4 d_c^2(W_1,W_2).
\end{equation}

We also need the following result \cite{CHS}.

\begin{lemma}
If $\cH_N$ is an $N$-dimensional Hilbert space and $P,Q$ are rank $M$
orthogonal projections onto subspaces $W_1,W_2$ respectively, then the 
chordal distance $d_c(W_1,W_2)$ between the subspaces satisfies
\[ d_c^2(W_1,W_2) = M - Tr\ PQ.\]
\end{lemma}

Next we give the precise connection between chordal distance for
subspaces and the
distance between the projections onto these subspaces.  This result can
be found in \cite{CHS} in the language of Hilbert-Schmidt norms.  We give our
own proof for the sake of completeness.

\begin{proposition}\label{prop3}
Let $\cH_N$ be an $N$-dimensional Hilbert space 
with orthonormal basis $\{e_i\}_{i=1}^N$.   Let $P,Q$ be the
orthogonal projections of $\cH_N$ onto $M$-dimensional subspaces
$W_1,W_2$ respectively.  Then the chordal distance between $W_1,W_2$ satisfies
\[ d_c^2(W_1,W_2) = \frac{1}{2}\sum_{i=1}^N\|Pe_i-Qe_i\|^2.\]
In particular, there are orthonormal bases $\{g_i\}_{i=1}^M$ for $W_1$
and $\{h_i\}_{i=1}^M$ for $W_2$ satisfying
\[ \frac{1}{2} \sum_{i=1}^N \|Pe_i-Qe_i\|^2 \le
\sum_{i=1}^M \|g_i-h_i\|^2 \le 2\sum_{i=1}^N \|Pe_i-Qe_i\|^2.\]
\end{proposition}

\begin{proof}
We compute:
\begin{eqnarray*}
 \sum_{i=1}^N\|Pe_i-Qe_i\|^2
&=& \sum_{i=1}^N \langle Pe_i-Qe_i,Pe_i-Qe_i\rangle\\
&=& \sum_{i=1}^N\|Pe_i\|^2 + \sum_{i=1}^N\|Qe_i\|^2 - 2 \sum_{i=1}^N\langle Pe_i,
Qe_i\rangle\\
&=& 2M -2\sum_{i=1}^N\langle PQe_i,e_i\rangle\\
&=& 2M -2 Tr\ PQ\\
&=& 2M - 2[M-d_c^2(W_1,W_2)]\\
&=& 2 d_c^2(W_1,W_2).
\end{eqnarray*}
This combined with Equation \ref{E17} completes the proof.
\end{proof}

Now we are ready to answer the second problem we need to address in this section.

\begin{theorem}\label{thm1}
Let $P,Q$ be projections of rank $M$ on $\cH_N$ and let $\{e_i\}_{i=1}^N$
be the natural orthonormal basis of $\cH_N$.  Further assume that there is
a Parseval frame $\{f_i\}_{i=1}^N$ for $\cH_M$ with analysis operator 
$T$ satisfying $Tf_i = Pe_i$, for all $i=1,2,\ldots,N$.   If 
\[ \sum_{i=1}^M \|Pe_i-Qe_i\|^2 < \epsilon,\]
then there is a Parseval frame $\{g_i\}_{i=1}^N$
for $\cH_M$ with analysis operator $T_1$ satisfying
\[ T_1g_i = Qe_i, \mbox{ for all }i=1,2,\ldots,N,\]
and 
\[ \sum_{i=1}^N \|f_i-g_i\|^2 <2\epsilon.
\]
Moreover, if $\{Qe_i\}_{i=1}^N$ is equal norm, then $\{g_i\}_{i=1}^N$ may be
chosen to be equal norm.
\end{theorem}

\begin{proof}
By Proposition \ref{prop3}, there are orthonormal bases
$\{a_j\}_{j=1}^M$ and $\{b_j\}_{j=1}^M$ for $W_1,W_2$ respectively satisfying
\[ \sum_{j=1}^M \|a_j-b_j\|^2 < 2 \epsilon.\]
Let $A,B$ be the $N \times M$ matrices whose $j^{th}$ columns are
$a_j,b_j$ respectively.  Let $a_{ij},b_{ij}$ be the $(i,j)$ entry of $A,B$ respectively.
Finally, let $\{f^{'}_i\}_{i=1}^N,\{g_i^{'}\}_{i=1}^N$ be the $i^{th}$ rows of $A,B$
respectively.  Then we have
\begin{eqnarray*}
\sum_{i=1}^N \|f_i^{'} - g_i^{'}\|^2 &=&
\sum_{i=1}^N \sum_{j=1}^M |a_{ij}-b_{ij}|^2 \\
&=& \sum_{j=1}^M \sum_{i=1}^N |a_{ij}-b_{ij}|^2\\
&=& \sum_{i=1}^M \|a_j-b_j\|^2\\
&\le& 2\epsilon.
\end{eqnarray*}
Since the columns of $A$ form an orthonormal basis for $W_1$, we know that 
$\{f_i^{'}\}_{i=1}^N$ is a Parseval frame which is isomorphic to $\{f_i\}_{i=1}^N$.
Thus there is a unitary operator $U:\cH_M \rightarrow \cH_M$ with
$Uf_i^{'} = f_i$.  Now let $\{g_i\}_{i=1}^N = \{Ug_i^{'}\}_{i=1}^N$.  Then
\[ \sum_{i=1}^N \|f_i-Ug_i^{'}\|^2 = \sum_{i=1}^N \|U(f_i^{'})-U(g_i^{'})\|^2
=\sum_{i=1}^N \|f_i^{'}-g_i^{'}\|^2\le 2\epsilon. \]
Finally, if $T_1$ is the analysis operator for the Parseval frame $\{g_i\}_{i=1}^N$, then
$T_1$ is a co-isometry and since
$\{T_1(g_i)\}_{i=1}^N  = \{Qe_i\}_{i=1}^N$, for all $i=1,2,\ldots,N$, if $Qe_i$ is equal norm,
so is $T_1(g_i)$ and hence so is $\{g_i\}_{i=1}^N$.
\end{proof}

\section{The Equivalence of our Problems}

Now we can show that the Paulsen Problem and the Projection Problem are
equivalent in the sense that their functions $f(\epsilon,M,N),\ g(\epsilon,M,N)$,
respectively, are equal up to a factor of two.

\begin{theorem}\label{thm2}
If $g(\epsilon,M,N)$ is the function for the Paulsen Problem and $f(\epsilon,M,N)$
is the function for the Projection Problem, then
\[       f(\epsilon,M,N) \le            4 g(\epsilon,M,N) \le  8f(\epsilon,M,N).
\] 
\end{theorem}

\begin{proof}
First, assume that Problem \ref{prob1} holds with function $f(\epsilon,M,N)$.
Let $\{f_i\}_{i=1}^N$ be a Parseval frame for $\cH_M$ satisfying
\[ (1-\epsilon)\frac{M}{N} \le \|f_i\|^2 \le (1+\epsilon)\frac{M}{N}.\]
Let $T$ be the analysis operator of $\{f_i\}_{i=1}^N$ and let $P$
be the projection of $\cH_N$ onto range $T$.  So, $Tf_i = Pe_i$, for all
$i=1,2,\ldots,N$. By our assumption that that Problem \ref{prob1}
holds, there is a projection $Q$ on $\cH_N$ with constant diagonal so that
\[ \sum_{i=1}^N \|Pe_i-Qe_i\|^2 \le f(\epsilon,M,N).
\]
By Theorem \ref{thm1}, there is a a Parseval frame $\{g_i\}_{i=1}^N$ for $\cH_M$
with analysis operator $T_1$ so that $T_1g_i = Qe_i$ and
\[ \sum_{i=1}^N \|f_i-g_i\|^2 \le 2 f(\epsilon,M,N).
\]
Since $T_1$ is a co-isometry and $\{T_1g_i\}_{i=1}^N$ is equal norm, it follows
that $\{g_i\}_{i=1}^N$ is an equal norm Parseval frame satisfying the Paulsen problem.

Conversely, assume the Parseval Paulsen problem has a positive solution
with function $g(\epsilon,M,N)$.  Let $P$ be
an orthogonal projection on $\cH_N$ satisfying
\[ (1-\epsilon)\frac{M}{N} \le \|Pe_i\|^2 \le (1+\epsilon)\frac{M}{N}.\]
Then $\{Pe_i\}_{i=1}^N$ is a Parseval frame for $\cH_M$ and by the Parseval Paulsen
problem, there is an equal norm Parseval frame
$\{g_i\}_{i=1}^N$ so that
\[ \sum_{i=1}^N\|f_i-g_i\|^2 < g(\epsilon,M,N).\]
Let $T_1$ be the analysis operator of $\{g_i\}_{i=1}^N$.
Letting $Q$ be the projection onto the range of $T_1$, we have that
$Qe_i= T_1g_i$, for all $i=1,2,\ldots,N$.
By Theorem \ref{T3}, we have that
\[ \sum_{i=1}^N\|Pe_i-T_1g_i\|^2= \sum_{i=1}^N\|Pe_i-Qe_i\|^2  \le 
4g(\epsilon,M,N).\]
Since $T_1$ is a co-isometry and $\{g_i\}_{i=1}^N$ is equal norm, it follows
that $Q$ is a constant diagonal projection.
\end{proof}

  \section{Generalizations of the Paulsen Problem}
  
In this section
 we will look at some recent generalizations of the Paulsen Problem.

\begin{definition}
We say a sequence of numbers $\{a_i\}_{i=1}^N$ is a {\bf Parseval 
admissible sequence}
for $\cH_M$ if there is a Parseval frame $\{f_i\}_{i=1}^N$ for $\cH_M$ 
satisfying
$\|f_i\|^2 = a_i^2$, for all $i=1,2,\ldots,N$.
\end{definition}

The following classification of Parseval
admissible sequences can be found in \cite{CFLKT}.

\begin{theorem}\label{thm100}
A sequence of numbers $\{a_i\}_{i=1}^N$ is a Parseval admissible sequence
for $\cH_M$ if and only if both of the following hold:

1.  $\sum_{i=1}^N a_i^2 = M$.

(2)  $a_i \le 1$, for every $i=1,2,\ldots,N$.
\end{theorem}

Now we give a generalization of the Paulsen Problem.

\begin{problem}[Generalized Paulsen Problem]
Let $\{a_i\}_{i=1}^N$ be a Parseval admissible sequence for 
$\cH_M$.  If $\{f_i\}_{i=1}^N$ is a Parseval frame for $\cH_M$ satisfying
\[ (1-\epsilon)a_i^2 \le \|f_i\|^2 \le (1+\epsilon)a_i^2,\]
find the closest Parseval frame $\{g_i\}_{i=1}^N$ 
satisfying:  $\|g_i\| = a_i$,
for all $i=1,2,\ldots,N$.
\end{problem}

The work in this paper can be re-done to show that the Generalized Paulsen
Problem is equivalent to a Generalized Projection Problem.

\begin{problem}[Generalized Projection Problem]
If $P$ is a rank $M$ orthogonal projection on $\ell_2(N)$,
$\{a_i\}_{i=1}^N$ satisfies Theorem \ref{thm100} and
\[ (1-\epsilon)a_i^2 \le \|Pe_i\|^2 \le (1+\epsilon)a_i^2,\]
find the closest projection $Q$ to $P$ satisfying $\|Qe_i\| = a_i$, for
all $i=1,2,\ldots,N$.
\end{problem}

We end with a further generalization of the Paulsen Problem to frame operators.

\begin{definition}
If $S$ is a positive, self-adjoint invertible operator on $\cH_M$, we say that a 
sequence of numbers $\{a_i\}_{i=1}^N$ is an $S${\bf -admissible sequence} if there exists
a frame $\{f_i\}_{i=1}^N$ for $\cH_M$ having $S$ as its frame operator and
so that $\|f_i\|^2 = a_i^2$, for all $i=1,2,\ldots,N$.
\end{definition}

The classification of $S$-admissible sequences goes back to Horn and Johnson \cite{HJ}.
The simplest proof of this result is due Casazza and Leon \cite{CL}.

\begin{theorem} 
Let $ S $ be a positive self-adjoint operator on a
$ N$-dimensional Hilbert space $H_N$. Let
$ \lambda_1 \ge \lambda_2 \ge \dots \lambda_N>0 $ 
be the eigenvalues of $S$. Fix $ M \ge N $ and real numbers 
$ a_1 \ge a_2 \ge \cdots \ge a_M > 0 $. The following are
equivalent:
\begin{itemize}
\item[(1)] There is a frame $ \{ \varphi_j \}_{j=1}^{M} $
for $H_N$ with frame operator $S$ and 
$ \norm{\varphi_j}=a_j $, for all
$j=1,2,\dots,M$.
\item[(2)] For every $ 1 \le k \le N $,
\begin{align} 
\sum_{i=1}^k a_i^2 \le
\sum_{i=1}^k \lambda_i ,\ \ \ \text{and} \ \ \
\sum_{i=1}^M a_i^2 =
\sum_{i=1}^N \lambda_i .
\end{align}
\end{itemize}
\end{theorem}

Our final generalization of the Paulsen Problem is:

\begin{problem}
If $S$ is a positive, self-adjoint invertible operator on $\cH_M$, $\{a_i\}_{i=1}^N$ is
an $S$-admissible sequence, and $\{f_i\}_{i=1}^N$ is a frame with frame operator $S$
and satisfying
\[ (1-\epsilon)a_i \le \|f_i\|^2 \le (1+\epsilon)a_i,\]
then find the closest frame $\{g_i\}_{i=1}^N$ so that $\|g_i\|^2 = a_i$, for all
$i=1,2,\ldots,N$.
\end{problem}

\section{The Paulsen Problem and Naimark Complements}

In this section we will use Naimark complements to show that we only
need to solve the Paulsen problem for $N \le 2M$.   If $\{f_i\}_{i=1}^N$ is a Parseval
frame for $\cH_M$ with analysis operator $T$ which is a co-isometry and satisfies
\[ Tf_i = Pe_i,\mbox{ for all }i=1,2,\ldots,N,\]
where $\{e_i\}_{i=1}^N$ is the natural orthonormal basis for $\ell_2(N)$ and
$P$ is the orthogonal projection of $\ell_2(N)$ onto $T(\cH_M)$, the
{\it Naimark complement} of $\{f_i\}_{i=1}^M$ is the Parseval frame
$\{(I-P)e_i\}_{i=1}^N$ for $\cH_{N-M}$.   Now we will compare the Paulsen
function for a Parseval frame to the Paulsen function for its Naimark complement.

\begin{theorem}
If $g(\epsilon,M,N)$ is the Paulsen constant then
\[ g(\epsilon,M,N) \le 8g(\epsilon \frac{M}{N-M},N-M,N).\]
\end{theorem}

\begin{proof}
Assume that $\cF=\{f_i\}_{i=1}^N$ is a $\epsilon$-nearly equal norm Parseval frame
for $\cH_N$ with analysis operator $T$ which is a co-isometry.
Then there is a  projection $P$ on $\ell_2(N)$ so that $Pe_i = Tf_i$, for all
$i=1,2,\ldots,N$.  It follows that $\{(I-P)e_i\}_{i=1}^N$ is a Parseval frame and
\begin{eqnarray*}
\|(I-P)e_i\|^2 &=& 1 - \|Pe_i\|^2\\
&\le& 1 - (1-\epsilon)\frac{M}{N}\\
&=& \left ( 1+\epsilon \frac{M}{N-M}\right )\left (1-\frac{M}{N}\right ).
\end{eqnarray*}
Similarly,
\[ \|(I-P)e_i\|^2 \ge \left ( 1-\epsilon \frac{M}{N-M}\right ) \left ( 1-\frac{M}{N}\right ).
\]
Choose a Parseval frame $\{g_i\}_{i=1}^N$ for $\cH_{N-M}$ with analysis operator
$T_1$ satisfying $T_1g_i = (I-P)e_i$.  Since $T_1$ is a co-isometry, it follows that
$\cG = \{g_i\}_{i=1}^N$ is a $\epsilon \frac{M}{N-M}$-nearly equal norm Parseval frame.
Hence, there is an equal norm Parseval frame $\cH=\{h_i\}_{i=1}^N$ for $\cH_{N-M}$
with
\[ d(\cG,\cH) \le g(\epsilon \frac{M}{N-M},N-M,N),\]
where $g$ is the Paulsen function for $N$ vectors in $\cH_{N-M}$.
Let $T_2$ be the analysis operator for $\cH$.  
Applying Theorem \ref{T3}, we have that  
\[ d(T_1(\cG),T_2(\cH)) \le 4 g(\epsilon \frac{M}{N-M},N-M,N).
\]
Let $I-Q$ be the orthogonal projection onto $T_2(\cH)$.  Now we check
\begin{eqnarray*}
d(\{Pe_i\}_{i=1}^N,\{Qe_i\}_{i=1}^N) &=& \sum_{i=1}^N\|Pe_i-Qe_i\|^2\\
&=& \sum_{i=1}^N \|(I-P)e_i - (I-Q)e_i\|^2\\
&=& d(T_1(\cG),T_2(\cH))\\
&\le& 4 g(\epsilon \frac{M}{N-M},N-M,N).
\end{eqnarray*}
By Theorem \ref{thm1}, we can choose a equal norm Parseval frame 
${\mathcal K}=\{k_i\}_{i=1}^N$ for $\cH_{N-M}$ with analysis
operator $T_3$ satisfying $T_3k_i = Qe_i$, for all $i=1,2,\ldots,N$
 and
\[ d(\cF,{\mathcal K}) \le 8g(\epsilon \frac{M}{N-M},N-M,N).\]
\end{proof}

Given $N\ge M$, then either $N\le 2M$ or $N \le 2(N-M)$.  So we have

\begin{corollary}
To solve the Paulsen problem, it suffices to solve it for Parseval frames
$\{f_i\}_{i=1}^N$ for $\cH_M$ with $N\le 2M$.
\end{corollary}


\end{document}